\documentclass[11pt,psamsfonts]{amsart}
\usepackage{amssymb}
\usepackage{amsmath}
\usepackage{graphicx}
\usepackage{amscd}
\usepackage{amsfonts}
\usepackage{amsbsy}
\usepackage[dvipsnames]{xcolor}
\usepackage[T1]{fontenc}
\usepackage[english]{babel}
\usepackage{enumerate}
\usepackage[colorlinks=true]{hyperref}
\usepackage{endfloat}
\hypersetup{urlcolor=blue, citecolor=blue,linkcolor=blue}

\def \inf{{\rm inf}}

\def \dim{{\rm dim \,}}
\def \dm{{\rm diam \,}}

\newcommand \N{\mathbb{N}}
\newcommand \F{\mathcal{F}}

\newcommand \St{{\rm St}}

\newcommand \ef{\mathbf{\Gamma}}

\newcommand \dih{\dim_{H}}
\newcommand \uno{\dim_{\ef}^{1}}
\newcommand \dos{\dim_{\ef}^{2}}
\newcommand \tres{\dim_{\ef}^{3}}
\newcommand \cuatro{\dim_{\ef}^{4}}
\newcommand \cinco{\dim_{\ef}^{5}}
\newcommand \seis{\dim_{\ef}^{6}}
\newcommand \h{\dim_{H}}
\newcommand \bc{\dim_{B}}

\newtheorem{prop}{Proposition}[section]
\newtheorem{teo}[prop]{Theorem}
\newtheorem{cor}[prop]{Corollary}

\newtheorem{defn}[prop]{Definition}
\newtheorem{obs}[prop]{Remark}

\theoremstyle{definition}
\newtheorem{ejem}{Example}
\theoremstyle{remark}

\newtheoremstyle{named}{}{}{\itshape}{}{\bfseries}{.}{.5em}{\thmnote{#3's }#1}
\theoremstyle{named}
\newtheorem*{namedtheorem}{Theorem}

\begin{document}
\title[Generalizing Moran's Theorem]{Generalizing Moran's Theorem}
\author[Fern\'andez-Mart\'{\i}nez, Guirao, S\'anchez-Granero]{M. Fern\'andez-Mart\'{\i}nez, Juan L.G. Guirao and M.A. S\'anchez-Granero}

\address{$^{1}$ Centro Universitario de la Defensa. Academia General del Aire.
Universidad Polit\'{e}cnica de Cartagena, 30720-Santiago de la
Ribera, Regi\'{o}n de Murcia, Spain}
\email{fmm124@gmail.com}

\address{$^{2}$ Departamento de Matem\'{a}tica Aplicada y Estad\'{\i}stica. Universidad
Polit\'{e}cnica de Cartagena, Hospital de Marina, 30203--Cartagena,
Regi\'{o}n de Murcia, Spain. --Corresponding Author--} \email{juan.garcia@upct.es}

\address{$^{3}$ Departamento de Matem\'{a}ticas. Universidad de Almer\'{\i}a. C/ Sacremento, S/N. 04120 La Ca\~{n}ada de San Urbano, Almer\'{\i}a, Spain } \email{misanche@ual.es}


\bigskip

\begin{abstract}
This work is aimed by the spirit of 1946 Moran's Theorem, which ensures that both the box and the Hausdorff dimensions for any attractor could be calculated as the solution of an equation involving only its similarity factors. To achieve such  result, the open set condition (OSC, herein) is required to be satisfied by the pieces of the attractor in order to guarantee that they do not overlap too much.
In this paper, we generalize the classical Moran's through a fractal dimension model which allows to obtain the fractal dimension of any IFS-attractor equipped with its natural fractal structure without requiring the OSC to be satisfied.
\end{abstract}

\maketitle

\section{Introduction}
The calculation of the fractal dimension for the attractor of an iterated function system (IFS-attractor, herein) arises as an interesting question in Fractal Theory which could be dealt with from the point of view of fractal structures. In this way, the classical result proved by Moran in 1946 (see \cite[Theorem III]{MOR46}), allows to affirm that both the box-counting and the Hausdorff dimensions for any strict self-similar set can be fully determined through a straightforward expression only involving the corresponding similarity factors. However, to achieve such a result, the OSC is required to be satisfied by the similarities in that IFS. Recall that the open set condition is a significant restriction for the pieces of an attractor to ensure that they do not overlap too much. 

In this paper, some attempts to obtain a parallel result to Moran's by means of fractal structures are contributed. In fact, both Theorems \ref{teo:8} and \ref{teo:dim3} have been shown for fractal dimensions II and III, respectively. Thus, while the former assumes that all the similarities must have a common similarity factor under the OSC to reach the equality between the box dimension and the fractal dimension II, the latter does not require to satisfy the OSC by the similarities that define the corresponding IFS-attractor. Moreover, it also allows different similarity factors in order to calculate its fractal dimension III by means of an equation of the form $\sum c_i^s=1$. To end the paper, an additional Moran type result linking all the fractal dimension models for a fractal structure considered along this paper is contributed.

The structure of this work remains as follows. In Section \ref{sec:pre}, all the necessary preliminaries, involving the basics on IFS-attractors, fractal structures (and in particular, the natural fractal structure for an IFS-attractor), the classical models for fractal dimension (namely, both the box and the Hausdorff dimensions), the fractal dimension models for a fractal structure we explore in this paper for IFS-attractors, the OSC, as well as the classical Moran's Theorem, are provided in order to make this paper self-contained. Further, Section \ref{sec:3} contains our contributed Moran type theorems. 

\section{Preliminaries}\label{sec:pre}

Along this paper, let $I=\{1,\ldots,k\}$ be a finite index set. In this section, we provide all the necessary concepts, definitions and theoretical results, in order to make the present work self-contained. Accordingly, the structure of this preliminary section is as follows. In Subsection \ref{sub:ifs}, we recall the classical notion of an IFS, which becomes essential to show our main theorems in forthcoming Section \ref{sec:3}. Additionally, in Subsection \ref{sub:fs}, the concept of fractal structure is sketched from a topological point of view. It is worth mentioning that, also in such subsection, the notion of natural fractal structure which any IFS can be equipped with, is introduced. The next two subsections contains the basic definitions to calculate the fractal dimension for a given space. Thus, while Subsection \ref{sub:classics} recalls the classical models for fractal dimension (namely, both the box dimension and the Hausdorff dimension), Subsection \ref{sub:fdfs} contains some models of fractal dimension for a fractal structure that have been developed following the spirit of classical fractal dimensions. Finally, to end this section, we recall the definition of the OSC as well as the 1946 Moran's Theorem in upcoming Subsections \ref{sub:osc} and \ref{sub:moran}, respectively.

\subsection{IFS-attractors.}\label{sub:ifs}

First, let $f:X\to X$ be a self-map defined on a metric space $(X,\rho)$. Recall that $f$ is said to be a Lipschitz self-map whenever it satisfies that $\rho(f(x),f(y))\leq c\, \rho(x,y)$, for all $x,y\in X$, where $c>0$ is called the Lipschitz constant associated with $f$. In particular, if $c<1$, then $f$ is said to be a contraction, and we will refer to $c$ as its contraction factor.
Further, if the equality in the previous expression is reached, namely, $\rho(f(x),f(y))=c\, \rho(x,y)$, for all $x,y\in X$, then $f$ is called a similarity, and its Lipschitz constant is also called as its similarity factor.

\begin{defn}
For a metric space $(X,\rho)$, let us define an IFS as a finite family $\F=\{f_i: i\in I\}$, where $f_i$ is a similarity, for all $i\in I$. The unique compact set $A\subset X$ which satisfies that $A=\bigcup_{f\in\F}f(A)$, is called the attractor of the IFS $\F$, or IFS-attractor, as well. It is also called a self-similar set.
\end{defn}

It is a standard fact from Fractal Theory that there exists an atractor for any IFS on a complete metric space $X$.

Fig. \ref{fig:sierp} contains a graphical representation for the so-called Sierpi\'nski gasket, which was first defined in \cite{SIER15}. An analytical approach to this IFS-attractor is provided next.

\begin{ejem}\label{ejem:1}
Let $I=\{1,2,3\}$ be a finite index set, and let $\F=\{f_i:i\in I\}$ be a finite set of similarities defined from the Euclidean plane into itself as follows:
\[ f_{i}(x,y) = \left\{ \begin{array}{lll}
         (\frac{x}{2},\frac{y}{2}) & \mbox{if $i=1$}\\
         f_{1}(x,y)+(\frac{1}{2},0) & \mbox{if $i=2$}\\
         f_{1}(x,y)+(\frac{1}{4},\frac{1}{2}) & \mbox{if $i=3$,}
        \end{array} \right. \]
for all $(x,y)\in \mathbb{R}^2$. Thus, the Sierpi\'nski gasket is fully determined as the unique non-empty compact subset $K$ satisfying the Hutchinson's equation $K=\bigcup_{i\in I}f_i(K)$. Note that each of its pieces $f_i(K)$ becomes a self-similar copy of the whole Sierpi\'nski gasket.
\end{ejem}

\begin{center}
\begin{figure}[h]
\centering
\includegraphics[width=65mm,height=55mm]{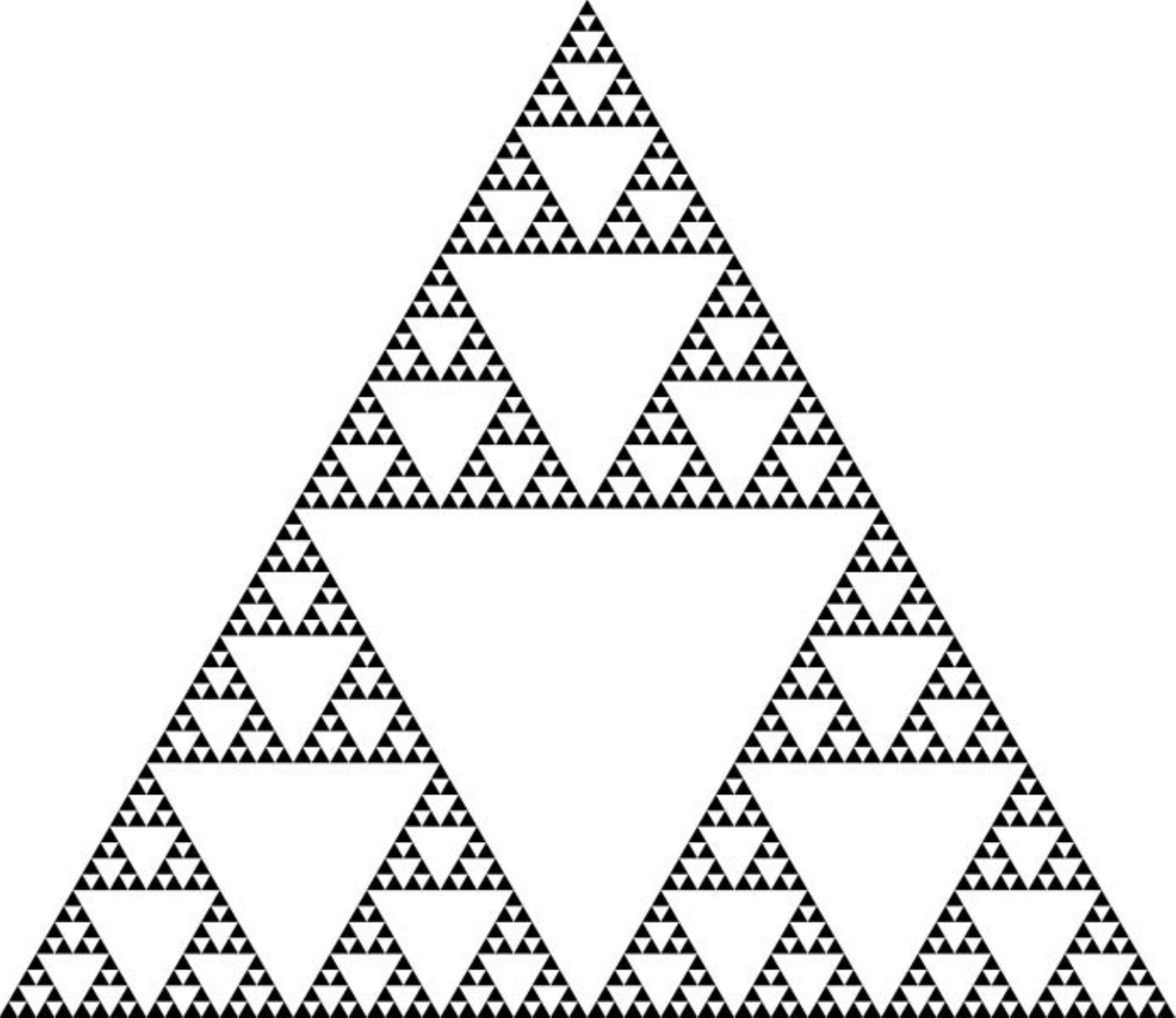}
\caption{Sierpi\'nski gasket.}\label{fig:sierp}
\end{figure}
\end{center}

\subsection{Fractal structures. The natural fractal structure for an IFS-attractor.}\label{sub:fs}

The concept of fractal structure, which naturally appears in several asymmetric topological topics \cite{SG10}, was first contributed in \cite{SG99A} to characterize non-Archimedeanly quasi-metrizable spaces. Moreover, in \cite{MSG12}, it was applied to deal with IFS-attractors.
On the other hand, fractal structures constitute a powerful tool to develop new fractal dimension models which allow to calculate the fractal dimension for a wide range of (non-Euclidean) spaces and contexts (see, e.g., \cite{MFM12}).

Recall that a family $\Gamma$ of subsets of a space $X$ is called a covering if $X=\bigcup\{A:A\in \Gamma\}$. 
A fractal structure is a countable collection of coverings of a given subset which provides better approximations to the whole space as deeper stages are reached, which we will refer to as \emph{levels} of the fractal structure.

Let $\Gamma_1$ and $\Gamma_2$ be any two coverings for $X$. Thus, $\Gamma_1\prec \Gamma_2$ means that $\Gamma_1$ is a \emph{refinement} of $\Gamma_2$, namely, for all $A\in \Gamma_1$, there exists $B\in \Gamma_2$ such that $A\subseteq B$. In addition to that, $\Gamma_1\prec \prec \Gamma_2$ denotes that $\Gamma_1\prec \Gamma_2$, and also, that for all $B\in \Gamma_2$, $B=\bigcup\{A\in \Gamma_1: A\subseteq B\}$. Hence, a fractal structure on a set $X$, is a countable family of coverings of $X$, $\ef=\{\Gamma_n:n\in \mathbb{N}\}$, such that $\Gamma_{n+1}\prec \prec \Gamma_n$, for all $n\in \mathbb{N}$. It is worth mentioning that covering $\Gamma_n$ is level $n$ of the fractal structure $\ef$. 

To simplify the theory, the levels of any fractal structure $\ef$ will not be coverings in the usual sense. Instead of this, we are going to allow that a set can appear more than once in any level of $\ef$. We would like also to point out that a fractal structure $\ef$ is said to be finite if all levels $\Gamma_n$ are finite coverings.


If $\ef$ is a fractal structure on $X$ and $\St(x,\ef)$ is a neighborhood base of $x$ for all $x\in X$, then we will call $\ef$ a starbase fractal structure. Starbase fractal structures are connected to metrizability (see, e.g., \cite{SG02A,SG02B}).
It is worth mentioning that a fractal structure $\ef$ is said to be finite if all levels $\Gamma_n$ are finite coverings. 

IFS-Attractors can be always equipped with a natural fractal structure, which was first sketched in \cite{BA92}, and formally defined later in \cite[Definition 4.4]{MSG12}. 
Next, we recall the description of such a fractal structure, which becomes essential for the upcoming section.

\begin{defn}\label{def:fsifs}
Let $\F$ be an IFS, whose associated IFS-attractor is $K$. The natural fractal structure on $K$ is given as the countable family of coverings $\ef=\{\Gamma_n: n\in \mathbb{N}\}$, where $\Gamma_n=\{f_{\omega}(K):\omega\in I^{n}\}$, for each $n\in \mathbb{N}$. Here, for a fixed natural number $n$ and each word $\omega=\omega_1 \ \omega_2 \ldots \ \omega_n\in I^{n}$, we denote $f_{\omega}=f_{\omega_1}\circ \ldots \circ f_{\omega_n}$.
\end{defn}

\begin{obs}
Another appropriate description of the levels of such a natural fractal structure is as follows: $\Gamma_1=\{f_i(K):i\in I\}$, and $\Gamma_{n+1}=\{f_i(A):A\in \Gamma_n,i\in I\}$, for all $n\in \mathbb{N}$.
\end{obs}

In Example \ref{ejem:1}, we described analytically that IFS whose associated attractor is the Sierpi\'nski gasket. Next, we describe the natural fractal structure that can be defined on this strict self-similar set.

\begin{ejem}
The natural fractal structure on the Sierpi\'nski gasket could be defined as the countable family of coverings $\ef=\{\Gamma_n:n\in \mathbb{N}\}$, where $\Gamma_1$ is the union of three equilateral ``triangles''\ whose sides are equal to $1/2$, $\Gamma_2$ consists of the union of $3^2$ equilateral ``triangles''\ with sides equal to $1/2^{2}$, and, in general, $\Gamma_n$ is the union of $3^n$ equilateral ``triangles''\ whose sides are equal to $1/2^n$, for all $n\in \N$. For illustration purposes, next Fig. \ref{fig:2} is provided. This contains a graphical approach to the first two levels in this fractal structure. Note that this is a finite fractal structure which also becomes starbase.
\end{ejem}

\begin{figure}[h]
\begin{center}
\begin{tabular}{cc}
\\ \includegraphics[width=60mm,height=46mm]{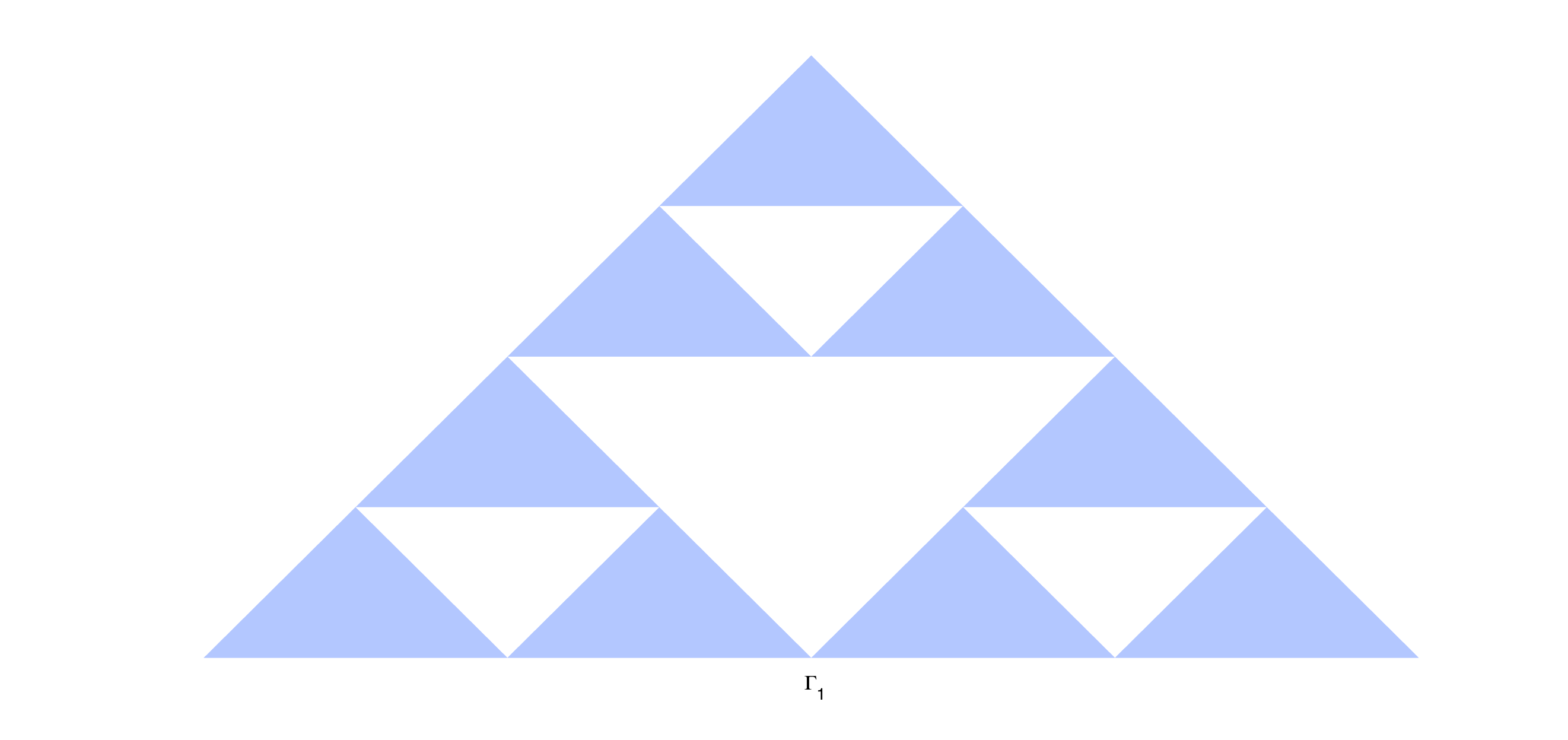} & \includegraphics[width=60mm,height=46mm]{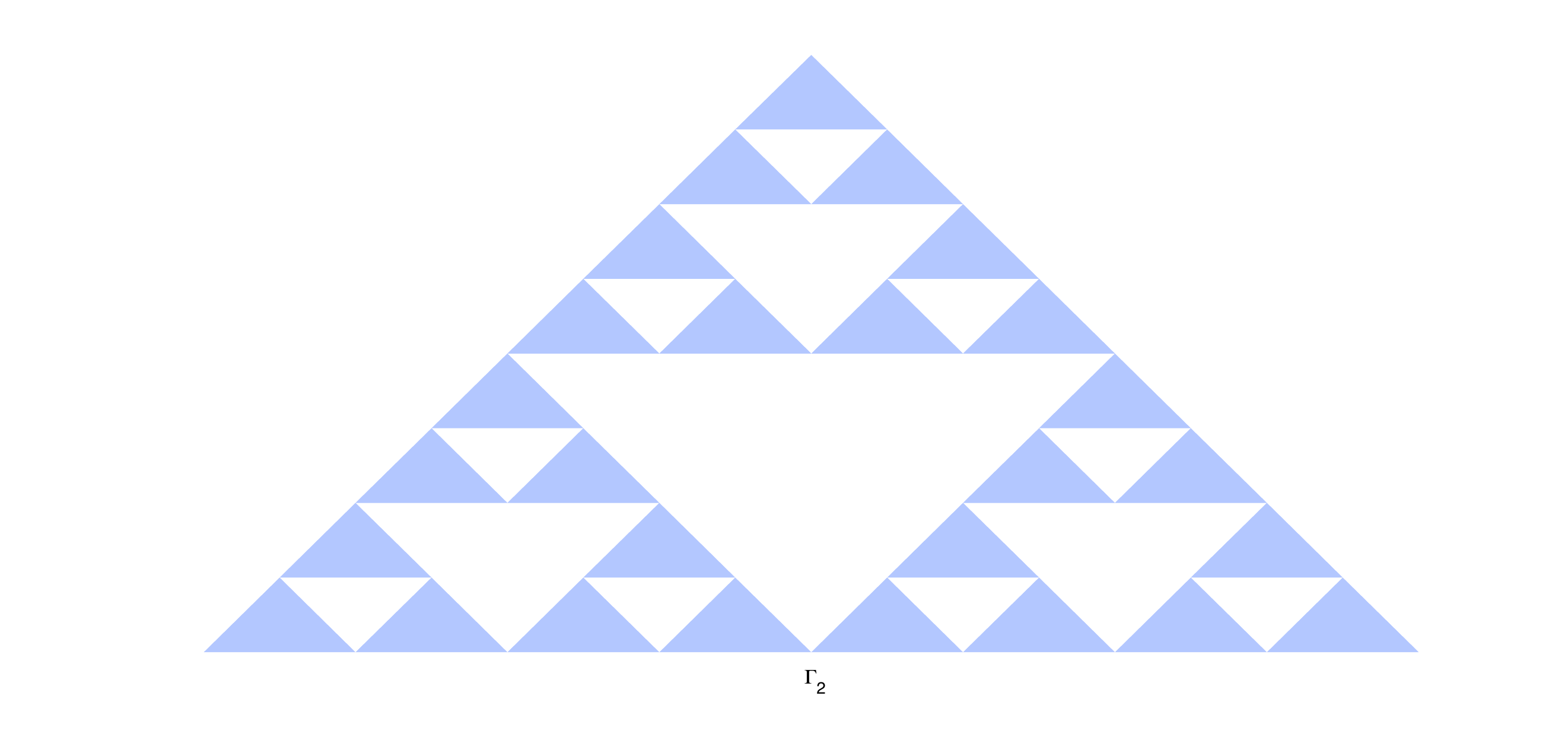}
\end{tabular}
\caption{First two levels for the natural fractal structure defined on the Sierpi\'nski gasket as an IFS-attractor.}\label{fig:2}
\end{center}
\end{figure}

\subsection{Classical models for fractal dimension.}\label{sub:classics}

Fractal dimension consists of a single quantity which yields valuable information about the complexity that a given space presents whether it is explored with enough level of detail. 

Next, we recall the definition of the standard box dimension, which is mainly used in empirical applications of fractal dimension due to the easiness of its empirical estimation.

\begin{defn}\label{def:bc}
The (lower/upper) box dimension for any subset $F\subseteq \mathbb{R}^d$ is given as the following (lower/upper) limit:
\begin{equation*}\label{eq:bc}
\bc(F)=\lim_{\delta\rightarrow 0}\frac{\log N_{\delta}(F)}{-\log \delta},
\end{equation*}
where $\delta$ is the scale, and $N_{\delta}(F)$ is
the largest number of disjoint balls of radii $\delta$ having centres in $F$.\label{eq:bc5}
\end{defn}
In 1919, Hausdorff used a method developed by Carath\'eodory some years earlier \cite{CAR14} in order to define the measures that now bear his name, and showed that the middle third Cantor set has positive and finite measure of dimension equal to $\log 2/ \log 3$ \cite{HAU19}. A detailed study about the analytical properties of both Hausdorff measure and dimension was mainly developed by Besicovitch and his pupils during the XXth century.


Next, let us recall the analytical construction of the Hausdorff dimension. Thus, let $(X,\rho)$ be a metric space, and let $\delta$ be a positive real number. For any subset $F$ of $X$, recall that a $\delta$-cover of $F$ is just a countable family of subsets $\{U_j\}_{j\in J}$ such that $F\subseteq \bigcup_{j\in J}U_j$, where $\dm(U_j)\leq \delta$, for all $j\in J$. Hence, let us denote by $\mathcal{C_\delta}(F)$ the collection of all $\delta$-covers of $F$. 
Moreover, let us consider the following quantity:
\begin{equation*}\label{eq:hdelta}
\mathcal{H}_\delta^s(F)=\inf\Bigg\{\sum_{j\in J}\dm(U_j)^s:\{U_j\}_{j\in J}\in \mathcal{C}_\delta(F)\Bigg\}.
\end{equation*}
We would like also to point out that the next limit always exists:
\begin{equation*}\label{eq:hm}
\mathcal{H}_H^s(F)=\lim_{\delta\rightarrow 0}\mathcal{H}_\delta^s(F),
\end{equation*}
which is called the $s$-dimensional Hausdorff measure of $F$. 
Hence, the Hausdorff dimension of $F$ is fully determined as the point $s$ where $\mathcal{H}_H^s(F)$ ``jumps" from $\infty$ to $0$, namely,
\begin{equation*}\label{eq:hreach}
\dih(F)=\inf\{s:\mathcal{H}_H^s(F)=0\}=\sup\{s:\mathcal{H}_H^s(F)=\infty\}.
\end{equation*}

\subsection{Fractal dimensions for a fractal structure.}\label{sub:fdfs}

In this subsection, we provide a theoretical description for all the fractal dimension models for a fractal structure that are applied in this paper to calculate the fractal dimension for IFS-attractors with respect to their natural structure. As we point out next, there are, basically, two kind of definitions: those which are similar to the classical box dimension, and on the other hand, those ones which become closer to the Hausdorff dimension. We encourage the reader to consult the following references for a deeper study regarding their theoretical properties: \cite{DIM1,DIM3,DIM4}.



\begin{defn}[Box dimension type models for a fractal structure]\label{}
Let $\ef$ be a fractal structure on a distance space $(X,\rho)$, $F$ be a subset of $X$, and $N_{n}(F)$ be the number of elements of $\Gamma_{n}$ which intersect $F$. Hence,
\begin{enumerate}[(1)]
\item the (lower/upper) fractal dimension I for $F$ is defined as the (lower/upper) limit:
\begin{equation*}
\uno(F)=\lim_{n\rightarrow \infty}\frac{\log N_{n}(F)}{n\log 2}.
\end{equation*}
\item the (lower/upper) fractal dimension II for $F$ is defined as the (lower/upper) limit:
\begin{equation*}
\dos(F)=\lim_{n\rightarrow \infty}\frac{\log N_{n}(F)}{-\log \delta(F,\Gamma_{n})}.
\end{equation*}
where $\delta(F,\Gamma_{n})=\sup\{\dm(A):A\in \Gamma_{n},A\cap F\neq \emptyset\}$ is the diameter of $F$ in level $n$ of the fractal structure.
\end{enumerate}
\end{defn}

\begin{defn}[Hausdorff dimension type models for a fractal structure]\label{}
Let $\ef$ be a fractal structure on a metric space $(X,\rho)$, $F$ be a subset of $X$, and assume that $\delta(F,\Gamma_n)\to 0$, and let us consider the following expressions:
\begin{enumerate}[(1)]
\item Given $n\in \N$: \label{eq:hnk}
$$\mathcal{H}_{n,k}^{s}(F)=\inf\Bigg\{\sum_{j\in J}\dm(A_j)^s:\{A_j\}_{j\in J}\in \mathcal{A}_{n,k}(F)\Bigg\},$$
where the family $\mathcal{A}_{n,k}(F)$ is given, in each case, as follows:
\begin{enumerate}[(i)]
\item $ \{\{A \in \Gamma_l: A \cap F \not=\emptyset \}: l \geq n \}$, \hfill if $k=3$; 
\item $\{\{A_{j}\}_{j\in J}:A_{j}\in \bigcup_{l\geq n}\Gamma_{l},\forall\, j\in J, F\subseteq \bigcup_{j\in J}A_{j}, |J|<\infty\}$,\hfill if $k=4$;
\item $\{\{A_{j}\}_{j\in J}:A_{j}\in \bigcup_{l\geq n}\Gamma_{l}, \forall\, j\in J, F\subseteq \bigcup_{j\in J}A_{j} \}$,\hfill if $k=5$,
\end{enumerate}     
where $|\cdot|$ denotes the number of elements in the corresponding set.   
Define also:
$$\mathcal{H}_k^s(F)=\lim_{n\rightarrow \infty}\mathcal{H}_{n,k}^{s}(F),$$
for $k\in \{3,4,5\}$, then the fractal dimension III (resp. IV, V) of $F$ is defined as the non-negative real such that
$$\dim_{\ef}^{k}(F)=\inf\{s:\mathcal{H}_k^s(F)=0\}=\sup\{s:\mathcal{H}_k^s(F)=\infty\}.$$
\item Consider the following expression:
$$\mathcal{H}_{\delta,6}^{s}(F)=\inf\Bigg\{\sum_{j\in J}\dm(A_j)^s:\{A_j\}_{j\in J}\in \mathcal{A}_{\delta,6}(F)\Bigg\},$$
where for each $\delta>0$, the family $\mathcal{A}_{\delta,6}(F)$ is defined by
$$\mathcal{A}_{\delta,6}(F)=\Bigg\{\{A_j\}_{j\in J}:A_j\in \bigcup_{l\in \mathbb{N}}\Gamma_l\ \forall \, j\in J, \dm(A_j)\leq \delta, F\subseteq \bigcup_{j\in J}A_j\Bigg\}.$$     
Thus, provided that
$$\mathcal{H}_6^s(F)=\lim_{\delta\rightarrow 0}\mathcal{H}_{\delta,6}^{s}(F),$$
then the fractal dimension VI of $F$ is defined as
$$\seis(F)=\inf\{s:\mathcal{H}_6^s(F)=0\}=\sup\{s:\mathcal{H}_6^s(F)=\infty\}.$$
\end{enumerate}
\item 
\end{defn}

\subsection{About the OSC}\label{sub:osc}

The OSC is a hypothesis required to the similarities $f_i$ of an Euclidean IFS $\F$ in order to guarantee that the pieces $f_i(K)$ of an IFS-attractor $K$ do not overlap \emph{too much}. Technically, such a condition is satisfied if and only if there exists a non-empty bounded open subset $V\subset \mathbb{R}^d$, such that $\bigcup_{i\in I}f_i(V)\subset V$, where that union remains disjoint (see, e.g., \cite[Section 9.2]{FAL90}).  

\subsection{Regarding the classical Moran's Theorem}\label{sub:moran}
In \cite[Theorem III]{MOR46} (or see \cite[Theorem 9.3]{FAL90}), it was provided a quite interesting result which allows the calculation of the box dimension of a certain class of Euclidean self-similar sets through the solution of an easy equation involving only a finite number of quantities, namely, the similarity factors that give rise to its corresponding IFS-attractor. Such a classical result is described next.


\begin{namedtheorem}[Moran](1946)
Let $\F$ be an Euclidean IFS satisfying the OSC and whose associated IFS-attractor is $K$. Let us suppose that $c_i$ is the similarity factor associated with each similarity $f_i$.
Then $\h(K)=\bc(K)=s$, where $s$ is given by
\begin{equation}\label{eq:fal}
\sum_{i=1}^{k}c_{i}^{s}=1.
\end{equation}
Further, for this value of $s$, it is verified that $\mathcal{H}_{H}^{s}(K)\in (0,\infty)$.
\end{namedtheorem}

It is worth mentioning that Moran also provided in \cite[Theorem II]{MOR46} a weaker version for the result above under the assumption that all the similarities are equal.

\section{Fractal dimension for self-similar sets}\label{sec:3}

The main goal in this section is to provide some Moran type results from the point of view of fractal structures. To deal with, firstly, in  Subsection \ref{teo:1}, we provide a weaker version for Moran's (also similar to \cite[Theorem II]{MOR46}), since under the OSC, this allows to calculate the fractal dimension of any IFS-attractor provided that all the similarities have a common similarity factor. Moreover, Subsection \ref{teo:2} contains a generalization for Moran's Theorem, since this will allow to calculate the fractal dimension for any IFS-attractor (equipped with its natural fractal structure) without requiring the OSC to be satisfied. Finally, Subsection \ref{teo:3} contains an additional result linking all the fractal structure based models to calculate the fractal dimension for any IFS-attractor which have been considered along this paper.

\subsection{A weaker version for Moran's Theorem}\label{teo:1}

It is well known that under the OSC, both the box and the Hausdorff dimensions for any IFS-attractor do agree. Even more, this value can be easily calculated from Eq. (\ref{eq:fal}). In this way, it would be an interesting task to show whether the box dimension and the fractal dimension II for any self-similar set under the OSC match. Moreover, if it occurs, the calculation of such quantity would become immediate from both the number of similarities as well as their common similarity factor. Notice also that the next theoretical result, which becomes a weaker version of Moran's Theorem, also provides an easy way to calculate the box dimension for any Euclidean IFS-attractor under the OSC.

\begin{teo}\label{teo:8}
Let $\F$ be an Euclidean IFS satisfying the OSC and whose associated IFS-attractor is $K$. Let $c\in (0,1)$ be the common similarity factor associated with each similarity $f_{i}$. Moreover, let $\ef$ be the natural fractal structure on $K$ as a self-similar set. Then,
\begin{equation}\label{eq:13}
\bc(K)=\dos(K)=\frac{-\log k}{\log c}.
\end{equation}
\end{teo}

\begin{proof}
First, to calculate the box dimension of $K$, let $N_{\delta}(K)$ be as in Definition \ref{def:bc}, and let us define $\delta_{n}=\delta(K,\Gamma_{n})$. Note that $\delta_{n}=c^{n}\, \dm(K)$, since $K$ is a strict self-similar set. By \cite[Lemma 4.18]{DIM1}, there are so many disjoint balls centered in $K$ and having radii $\varepsilon_{n}=c^{n}\, \varepsilon$, where $\varepsilon>0$, as the number of elements in $I^{n}$. On the other hand, since $N_{\varepsilon_{n}}(K)$ is the largest number of such balls, it becomes clear that the number of elements in $I^{n}$ is at most $N_{\varepsilon_{n}}(K)$. Thus, 
$N_{n}(K)\leq N_{\varepsilon_{n}}(K)$,
for all $n\in \mathbb{N}$. Moreover, we affirm that there exists $\gamma>0$ such that $\delta(K,\Gamma_{n})=\gamma\, \varepsilon_{n}$. In fact, this is true for $\gamma=\dm(K)/\varepsilon$. Overall, we can state that
\begin{equation}\label{eq:12}
\overline{\dim}_{\ef}^{2}(K)\leq\overline{\lim}_{n\rightarrow \infty}\frac{\log N_{\varepsilon_{n}}(K)}{-\log \varepsilon_{n}}=\overline{\dim}_{B}(K).
\end{equation}
Accordingly, the following chain of inequalities holds:
$$\underline{\dim}_{B}(K)\leq \underline{\dim}_{\ef}^{2}(K)\leq \overline{\dim}_{\ef}^{2}(K)\leq \overline{\dim}_{B}(K),$$
where the first inequality is due to \cite[Corollary 4.12 (1)]{DIM1}, and the last one is ensured by Eq. (\ref{eq:12}). Further, the existence of the box dimension for $K$ implies the existence of the fractal dimension II for $K$, and hence, we get the desired equality: $\bc(K)=\dos(K)$. Finally, just apply the classical Moran's Theorem to get the last equality in Eq. (\ref{eq:13}). Indeed, the expression
$\sum_{i\in I}c_{i}^{s}=k\, c^{s}=1$
leads to $s=-\log k/\log c$.
\end{proof}

\subsection{Generalizing the classical Moran's Theorem}\label{teo:2}

The main goal in this subsection is to show that the fractal dimension for any IFS-attractor could be calculated by means of an Eq. (\ref{eq:fal}) type, involving only a finite number of known quantities, even if the corresponding similarities do not satisfy the OSC. Thus, this kind of result generalizes the classical Moran's Theorem in the context of fractal structures. The proof for such a theoretical result is based on the natural fractal structure  which each IFS-attractor can be equipped with, and also takes into account Definition (\ref{eq:hnk})(i) for $\mathcal{H}_{n,3}^s$.

\begin{teo}\label{teo:dim3}
Let $\F$ be an IFS on a complete metric space $X$, whose associated IFS-attractor is $K$. Let us suppose that $c_i$ is the similarity factor associated with each similarity $f_i$, and let $\ef$ be the natural fractal structure on $K$ as a self-similar set. Then, $\tres(K)=s$, where $s$ is given as the solution of the following equation: 
$$\sum_{i\in I}c_i^s=1.$$
Moreover, for that value of $s$, it is satisfied that $\mathcal{H}_3^s(K)\in (0,\infty)$.
\end{teo}

\begin{proof}
First, recall that $K$ is the unique non-empty compact subset of $X$ which verifies the Hutchinson equation $K=\bigcup_{i\in I}f_i(K)$. This allows us to affirm that $\mathcal{A}_{n,3}(K)=\{\Gamma_m:m\geq n\}$, for each $n\in \mathbb{N}$. On the other hand, let $s\geq 0$ be such that $\sum_{i\in I}c_i^s=1$, and let us also define $J_l=\{(i_1,\ldots,i_l):i_j\in I, 1\leq j\leq l\}$. Thus, if we denote $K_{i_1\ldots\, i_l}=f_{i_1}\circ \ldots \circ f_{i_l}(K)$, then we get $K=\bigcup_{J_l}K_{i_1\ldots\, i_l}$. Moreover, it becomes straightforward to show that $c_{i_1}\ldots\, c_{i_l}$ is the similarity factor associated with the composition of similarities $f_{i_1}\circ f_{i_2}\circ \ldots \circ f_{i_l}$, and hence $\dm(K_{i_1\ldots\, i_l})=c_{i_1} \cdots c_{i_l} \dm(K)$. Note also that $\sum_{(i_1, \ldots,i_l) \in J_l} c_{i_1}^s \cdots c_{i_l}^s=\sum_{i_1 \in I} c_{i_1}^s\, \ldots\, \sum_{i_m \in I} c_{i_l}^s=1$. Accordingly,
\begin{equation*}
\begin{split}
\mathcal{H}_{n,3}^s(K)
&=\inf\Bigg\{\sum_{A \in \Gamma_m}\dm(A)^s:m\geq n\Bigg\}\\
&=\inf\Bigg\{\sum_{(i_1,\ldots,i_m)\in J_m} \dm(K_{i_1\ldots\, i_m})^s: m\geq n\Bigg\}\\
&=\inf\Bigg\{\sum_{i_1 \in I} c_{i_1}^s\, \ldots\, \sum_{i_m \in I} c_{i_m}^s\, \dm(K)^s: m\geq n\Bigg\},
\end{split}
\end{equation*}
for each $n\in \N$. This implies that $\mathcal{H}_3^s(K)=\dm(K)^s$. Hence, $\mathcal{H}_3^s(K)\notin \{0,\infty\}$, so $s=\tres(K)$.
\end{proof}

It is worth mentioning that both Theorem \ref{teo:dim3} and Moran's Theorem lead to an interesting result whose proof becomes now immediate: under the OSC, both the box and the Hausdorff dimensions are equal to the fractal dimension III.

\begin{cor}\label{cor:dim3}
Let $\F$ be an Euclidean IFS satisfying the OSC and whose associated IFS-attractor is $K$. Moreover, let $\ef$ be the natural fractal structure on $K$ as a self-similar set. Then,
$\tres(K)=\bc(K)=\h(K)$.
\end{cor}

\subsection{An additional Moran type theorem}\label{teo:3}

The following theorem we present next provides an additional Moran type theorem in the context of fractal structures. Indeed, it establishes that under the OSC, all the fractal dimension models involved in this paper are equal to both the Hausdorff and the box dimension, and they can be also explicitly calculated through an expression like Eq. (\ref{eq:fal}). Moreover, if all the similarity factors are equal, then additional connections within both fractal dimensions I \& II could be stated. It is worth mentioning that all the fractal dimensions for a fractal structure that appear in this result are determined with respect to the natural fractal structure associated with the corresponding IFS-attractor (recall Definition \ref{def:fsifs}).

\begin{teo}
Let $\F$ be an Euclidean IFS satisfying the OSC and whose associated IFS-attractor is $K$.  Let $c_i$ be the similarity factor associated with each similarity $f_i$ Further, let $\ef$ be the natural fractal structure on $K$ as a self-similar set. The two following hold:
\begin{enumerate}[(1)]
\item $\bc(K)=\tres(K)=\cuatro(K)=\cinco(K)=\seis(K)=\dih(K)=s$.
\item If all the similarity factors are equal, namely, $c_i=c\in (0,1)$, for each $i\in I$, then 
$\bc(K)=\dos(K)=\tres(K)=\cuatro(K)=\cinco(K)=\seis(K)=\dih(K)\leq \gamma_c\cdot \uno(K)$, where $\gamma_c=-\log 2/\log c$. 
\end{enumerate}
Moreover, $s$ is given by
\begin{equation}\label{eq:falnew}
\sum_{i\in I}c_i^s=1,
\end{equation}
and for that value of $s$, one gets that $\mathcal{H}_H^s(K),\mathcal{H}_m^s(K)\in (0,\infty)$, for $m=3,4,5,6$.
\end{teo}

\begin{proof}
\begin{enumerate}[(1)]
\item First, note that \cite[Proposition 3.5 (3)]{DIM4} can be applied, since the natural fractal structure on $K$ as a self-similar set is finite. Thus, $\dih(K)\leq \seis(K)\leq \cinco(K)\leq \cuatro(K)\leq \tres(K)$. On the other hand, Theorem \ref{teo:dim3} leads to $\tres(K)=s$, where $s$ is the solution of Eq. (\ref{eq:falnew}). Finally, Corollary \ref{cor:dim3} gives the result, since it implies that $\tres(K)=\bc(K)=\dih(K)$.
\item Let us apply \cite[Proposition 3.5 (3)]{DIM4} to get $\dih(K)\leq \seis(K)\leq \cinco(K)\leq \cuatro(K)\leq \tres(K)\leq \dos(K)$. Furthermore, notice that Moran's Theorem allows to affirm that $\bc(K)=\dih(K)$. Finally, recall that Theorem \ref{teo:8} gives that $\dos(K)=\bc(K)$, since all the similarity factors are equal. Note also that \cite[Corollary 3.10]{DIM1} leads to $\bc(K)\leq \gamma_c\cdot \uno(K)$, where $\gamma_c=-\log 2/ \log c$.
\end{enumerate}
In addition to that, the following chain of inequalities is satisfied:
$$\mathcal{H}_H^s(K)\leq \mathcal{H}_6^s(K)\leq \mathcal{H}_5^s(K)\leq \mathcal{H}_4^s(K)\leq \mathcal{H}_3^s(K),$$
where $\mathcal{H}_H^s(K)>0$ (just applying Moran's Theorem), and $\mathcal{H}_3^s(K)<\infty$, due to Theorem \ref{teo:dim3}. 
\end{proof}

\section*{Acknowledgements}
The first author is specially grateful to
Centro Universitario de la Defensa en la Academia General del
Aire de San Javier (Murcia, Spain). The second author of this work was partially supported by
MICINN/ FEDER grant number MTM2011--22587, MINECO grant number MTM2014-51891-P and Fundaci\'{o}n S\'{e}neca de la Regi\'{o}n de Murcia grant number 19219/PI/14. The third author acknowledges the support of the Ministry of Economy and Competitiveness of Spain, Grant MTM2012-37894-C02-01.

\end{document}